\documentclass[twoside, 12pt]{article}
\usepackage{amssymb, amsmath, mathrsfs, amsthm}
\usepackage{graphicx}
\usepackage{color}
\usepackage{cite}
\usepackage[top=2cm, bottom=2cm, left=2cm, right=2cm]{geometry}
\usepackage{float, caption, subcaption}

\input{epsf}

\newcommand{\bd}{\begin{description}}
\newcommand{\ed}{\end{description}}
\newcommand{\bi}{\begin{itemize}}
\newcommand{\ei}{\end{itemize}}
\newcommand{\be}{\begin{enumerate}}
\newcommand{\ee}{\end{enumerate}}
\newcommand{\beq}{\begin{equation}}
\newcommand{\eeq}{\end{equation}}
\newcommand{\beqs}{\begin{eqnarray*}}
\newcommand{\eeqs}{\end{eqnarray*}}

\definecolor{DarkGreen}{rgb}{0.2, 0.6, 0.3}

\newtheorem{theorem}{Theorem}[section]

\newtheorem{corollary}[theorem]{Corollary}

\newtheorem{claim}{Claim}[theorem]
\newtheorem{fact}{Fact}

\newtheorem{proposition}{Proposition}[section]

\newtheorem{observation}{Observation}[section]
\begin{document}
\title{\textbf{Fractional matching preclusion number of graphs} \footnote{Supported by the National
Science Foundation of China (Nos. 11601254, 11551001, 11161037,
61763041, 11661068, and 11461054) and the Science Found of Qinghai
Province (Nos.  2016-ZJ-948Q, and 2014-ZJ-907) and the  Qinghai Key
Laboratory of Internet of Things Project (2017-ZJ-Y21).} }

\author{
Jinyu Zou\footnote{School of Computer Science, Qinghai Normal
University, Xining, Qinghai 810008, China. {\tt
zjydjy2015@126.com}}, \ \ Yaping Mao\footnote{Corresponding author}
\footnote{School of Mathematics and Statistis, Qinghai Normal
University, Xining, Qinghai 810008, China. {\tt
maoyaping@ymail.com}} \footnote{Academy of Plateau Science and
Sustainability, Xining, Qinghai 810008, China}, \ \ Zhao
Wang\footnote{College of Science, China Jiliang University, Hangzhou
310018, China. {\tt wangzhao@mail.bnu.edu.cn}}, \ \ Eddie Cheng
\footnote{Department of Mathematics and Statistics, Oakland
University, Rochester, MI USA 48309 {\tt echeng@oakland.edu}}}
\date{}
\maketitle

\begin{abstract}
The \emph{fractional matching preclusion number} of a graph $G$, denoted by
$fmp(G)$, is the minimum number of edges whose deletion results in a
graph that has no fractional perfect matchings. In this paper, we first give some sharp upper and lower
bounds of fractional matching preclusion number. Next, graphs with large and
small fractional matching preclusion number are characterized, respectively. In
the end,
we investigate some extremal problems on fractional matching preclusion number. \\[2mm]
{\bf Keywords:} Interconnection networks; fractional perfect
matching;
fractional matching number; extremal problem\\[2mm]
{\bf AMS subject classification 2010:} 05C40; 05C05; 05C76.
\end{abstract}

\section{Introduction}

All graphs considered in this paper are undirected, finite and
simple. We refer to the book \cite{Bondy} for graph theoretical
notation and terminology not described here. For a graph $G$, let
$V(G)$, $E(G)$, and $\overline{G}$ denote the set of vertices, the
set of edges, and the complement of $G$, respectively. The number of
vertices in $G$ is the \emph{order} of $G$. For any subset $X$ of
$V(G)$, let $G[X]$ denote the subgraph induced by $X$; similarly,
for any subset $F$ of $E(G)$, let $G[F]$ denote the subgraph induced
by $F$. Let $X\subseteq V(G)\cup E(G)$. We use $G-X$ to denote the
subgraph of $G$ obtained by removing all the vertices in $X$
together with the edges incident with them from $G$ as well as
removing all the edges in $X$ from $G$. If $X=\{x\}$, we may write
$G-x$ instead of $G-\{x\}$. For two subsets $X$ and $Y$ of $V(G)$ we
denote by $E_G[X,Y]$ the set of edges of $G$ with one end in $X$ and
the other end in $Y$. The {\it degree}\index{degree} of a vertex $v$
in a graph $G$, denoted by $deg_G(v)$, is the number of edges of $G$
incident with $v$. Let $\delta(G)$ and $\Delta(G)$ be the minimum
degree and maximum degree of the vertices of $G$, respectively. The
set of neighbors of a vertex $v$ in a graph $G$ is denoted by
$N_G(v)$. A graph is \emph{Hamiltonian} if it contains a Hamiltonian
cycle. A component of a graph is \emph{odd} or \emph{even} according
to whether it has an odd or even number of vertices.

\subsection{Matching preclusion}

A \emph{matching} in a graph is a set of edges such that every
vertex is incident with at most one edge in this set. If a set of
edges form a matching in a graph, they are \emph{independent}. A
\emph{perfect matching} in a graph is a set of edges such that every
vertex is incident with exactly one edge in this set. An
\emph{almost-perfect matching} in a graph is a set of edges such
that every vertex, except one, is incident with exactly one edge in
this set, and the exceptional vertex is incident to none. So if a
graph has a perfect matching, then it has an even number of
vertices; if a graph has an almost-perfect matching, then it has an
odd number of vertices. The \emph{matching preclusion number} of a
graph $G$, denoted by $mp(G)$, is the minimum number of edges whose
deletion leaves the resulting graph with neither perfect matchings
nor almost-perfect matchings. Such an optimal set is called an
\emph{optimal matching preclusion set}. We define $mp(G)=0$ if $G$
has neither perfect matchings nor almost-perfect matchings. This
concept of matching preclusion was introduced in \cite{BrighamHVY}
and further studied in \cite{BrighamHVY, ChengLJ, ChengLJa, ChengCM,
ChengL, ChengL2, ChengHJL, ChengLLL, Park, ParkS, ParkI, ParkI2,
WangWLL, MWCM, WangFZ}. Originally this concept was introduced as a
measure of robustness in the event of edge failure in
interconnection networks. An interconnection network with a larger
$MP$ number may be considered as more robust in the event of link
failures.

\begin{proposition}{\upshape \cite{BrighamHVY}}\label{pro1-1}
Let $G$ be a graph with an even number of vertices. Then $mp(G)\leq
\delta(G)$, where $\delta(G)$ is the minimum degree of $G$.
\end{proposition}

The following result, due to Dirac, is well-known.

\begin{theorem}{\upshape Dirac \cite{Bondy} \ (p-485)}\label{thA}
Let $G$ be a simple graph of order $n \ (n\geq 3)$ and minimum
degree $\delta$. If $\delta\geq \frac{n}{2}$, then $G$ is
Hamiltonian.
\end{theorem}

\begin{theorem}{\upshape \cite{Ore}}\label{th1-2}
Let $G$ be a graph of order $n$ in which $deg_{G}(u)+deg_{G}(v)\geq n$ for each pair of nonadjacent vertices $u,v$. Then $G$ is Hamiltonian.
\end{theorem}

\subsection{Fractional matching preclusion}

A fractional matching is a function $f$ that assigns to each edge a
number in $[0,1]$ such that $\sum_{e\thicksim v}f(e)\leq 1$ for each
vertex $v$, where the sum is taken over all edges $e$ incident with
$v$. If $f(e)\in \{0,1\}$ for each edges $e$, then $f$ is just a
matching. Clearly,
$$
\sum_{e\in E(G)}f(e)=\frac{1}{2}\sum_{v\in V(G)}\sum_{e\thicksim
v}f(e)\leq\frac{|V(G)|}{2}.
$$

The \emph{fractional matching number} of $G$, denoted by
$\mu_{f}(G)$, is the supremum of $\sum_{e\in E(G)}f(e)$ over all
fractional matching $f$. A \emph{fractional perfect matching} is a
fractional matching $f$ satisfying that $\sum_{e\in E(G)}f(e)=1$ for
every $v\in V(G)$. Clearly, a fractional matching $f$ is perfect if
and only if $\sum f(e)=\frac{|V(G)|}{2}$ and a perfect matching is a
fractional perfect matching.

An edge subset $F$ of $G$ is a \emph{\emph{fractional MP set}}
(\emph{FMP set for short}) if $G-F$ has no fractional perfect
matchings. The \emph{FMP number} of $G$, denoted by $fmp(G)$, is the
minimum size of FMP sets of $G$, that is,
$$
fmp(G)=\min\{|F|:F \text{ is an FMP set}\}.
$$

The following results are immediate.

\begin{observation}\label{obs1-1}
$(1)$ If $H$ is a spanning subgraph of $G$, then $fmp(H)\leq
fmp(G)$.

$(2)$ If $e$ is an edge of $G$, then $fmp(G-e)\geq fmp(G)-1$.

$(3)$ For an even graph $G$, $mp(G)\leq fmp(G)$.

$(4)$ If $|V(G)|$ is even and $mp(G)=\delta(G)$, then
$fmp(G)=mp(G)=\delta(G)$.

$(5)$ If $G=H_1\cup H_2$, then $fmp(G)=\min\{fmp(H_1),fmp(H_2)\}$.
\end{observation}

For complete graphs, Liu and Liu \cite{LiuLiu} derived the
following result.

\begin{theorem}{\upshape \cite{LiuLiu}}\label{th1-3}
For a complete graph $K_n$, $fmp(K_{n})=n-1$ if $n\geq 7$.
\end{theorem}

\subsection{Extremal problem}

One of the interesting problems in extremal graph theory is the
Erd\"{o}s-Gallai-type problem, which is to determine the maximum or
minimum value of a graph parameter with some given properties. In
\cite{CaiLiWu, JiangLiZhang}, the authors investigated two kinds of
Erd\"{o}s-Gallai-type problems for monochromatic connection number
and monochromatic vertex connection number, respectively. Motivated
by these, we study two kinds of Erd\"{o}s-Gallai-type problems for
$fmp(G)$ in this paper.

\begin{itemize}
\item[] \noindent {\bf Problem 1.} Given two positive integers $n$ and $k$,
compute the minimum integer $f(n,k)$ such that for every connected
graph $G$ of order $n$, if $|E(G)|\geq f(n,k)$ then $fmp(G)\geq k$.

\item[] \noindent {\bf Problem 2.} Given two positive integers $n$ and $k$,
compute the maximum integer $g(n,k)$ such that for every graph $G$
of order $n$, if $|E(G)|\leq g(n,k)$ then $fmp(G)\leq k$.
\end{itemize}

Another interesting problem in extremal graph theory is to study the
minimum size of graphs with given parameter; see \cite{Schiermeyer}.
\begin{itemize}
\item[] \noindent {\bf Problem 3.} Given two positive integers $n$ and
$k$, compute the minimum integer $s(n,k)=\min\{|E(G)|:G\in
\mathscr{G}(n,k)\}$, where $\mathscr{G}(n,k)$ the set of all graphs
of order $n$ (that is, with $n$ vertices) with fractional matching
preclusion number $k$.
\end{itemize}

In Section 2, we show that $0\leq fmp(G)\leq n-1$ for a graph $G$ of
order $n$, and the graphs with $fmp(G)=0,1$, odd graphs
with $fmp(G)=n-1,n-2,n-k$ are characterized. In Section 3, we study
the above extremal problems on fractional matching preclusion
number. The results in this paper can be viewed as the fractional
matching preclusion number analogues of those in
\cite{WangMelekianChengMao}. So the basic structure of some of the
proofs are similar. However, more analysis are required here.

\section{Graphs with given FMP number}

Let $o(G)$ be the number of components of $G$ with an odd
number of vertices and $i(G)$ be the number of isolated
vertices of $G$. The main theorem on perfect matchings is the
following, due to Tutte \cite{Tutte}.

\begin{theorem}{\upshape \cite{Tutte}}\label{th2-1}
A graph $G$ has a perfect matching if and only if $o(G-S)\leq |S|$
for every subset $S\subseteq V(G)$.
\end{theorem}

The analogous theorem for fractional perfect matchings is the
following result.

\begin{theorem}{\upshape \cite{ScheinermanUllman}}\label{th2-2}
A graph $G$ has a fractional perfect matching if and only if
$i(G-S)\leq |S|$ for every subset $S\subseteq V(G)$.
\end{theorem}

It is clear that the above theorem can be used to characterized
graph with fractional preclusion number at most $k$ in the following
way: $fmp(G)\leq k$ if and only if there exist $T\subset E(G)$ where
$|T|\leq k$ and $S\subseteq V(G-T)$ such that $i(G-S\cup T)>|S|$.

\begin{theorem}{\upshape \cite{ScheinermanUllman}}\label{th2-3}
The following are equivalent for a graph $G$.

$(1)$ $G$ has a fractional perfect matching.

$(2)$ There is a partition ${V_{1},V_{2},\cdots,V_{n}}$ of the
vertex set $V(G)$ such that, for each $i$, the graph $G[V_{i}]$ is
either $K_{2}$ or Hamiltonian.

$(3)$ There is a partition ${V_{1},V_{2},\cdots,V_{n}}$ of the
vertex set $V(G)$ such that, for each $i$, the graph $G[V_{i}]$ is
either $K_{2}$ or a Hamiltonian graph on an odd number of vertices.
\end{theorem}

\begin{proposition}\label{pro2-1}
Let $G$ be a graph of order $n$. If $fmp(G)=k$, then
$$
\delta(G)\leq \frac{n}{2}+k-1.
$$
\end{proposition}
\begin{proof}
Suppose $\delta(G)\geq \frac{n}{2}+k$. For any $X\subseteq E(G)$ and
$|X|=k$, $\delta(G-X)\geq \frac{n}{2}$. From Theorem \ref{thA}, $G-X$
contains a Hamiltonian cycle, and hence $G-X$ contains a fractional
perfect matching, which contradicts the fact that $fmp(G)=k$.
\end{proof}

Note that each graph $G$ with $n$ vertices is a spanning subgraph of
$K_{n}$. The following bounds are immediate by Observation
\ref{obs1-1} and Theorem~\ref{th1-3}.

\begin{proposition}\label{pro2-2}
Let $G$ be a connected graph of order $n$. Then
$$
0\leq fmp(G)\leq n-1.
$$
Moreover, the upper and lower bounds are sharp.
\end{proposition}

\subsection{Graphs with small FMP number}

The following corollary is immediate by Theorem \ref{th2-3}.
\begin{corollary}\label{cor2-4}
Let $G$ be a graph. Then $fmp(G)=0$ if and only if for every partition
${V_{1},V_{2},\cdots,V_{n}}$ of the vertex set $V(G)$, there exists
some $V_{j}$ such that the graph $G[V_{i}]$ is neither $K_{2}$ nor a
Hamiltonian graph of order $n$.
\end{corollary}

Connected graphs with $fmp(G)=1$ can be characterized in the
following.
\begin{proposition}\label{pro2-3}
Let $G$ be a connected graph. Then $fmp(G)=1$ if and only if $G$
satisfies the following conditions.

$(1)$ There is a partition ${V_{1},V_{2},\cdots,V_{n}}$ of the
vertex set $V(G)$ such that, for each $i$, the graph $G[V_{i}]$ is
either $K_{2}$ or a Hamiltonian graph on an odd number of vertices.

$(2)$ There exists an edge $e\in E(G)$ such that for every partition
${V_{1},V_{2},\cdots,V_{n}}$ of the vertex set $V(G)$, there exists
some $V_{j}$ such that the graph $G'[V_{i}]$ is neither $K_{2}$ nor
a Hamiltonian graph of order $n$, where $G'=G-e$.
\end{proposition}
\begin{proof}
Suppose $(1)$ and $(2)$ hold. Since $(1)$ holds, it follows from
Theorem \ref{th2-3} that $fmp(G)\geq 1$. Since $(2)$ holds, it
follows that $fmp(G')=0$ and hence $fmp(G)\leq fmp(G')+1=1$. So
$fmp(G)=1$.

Suppose $fmp(G)=1$. From Theorem \ref{th2-3}, $(1)$ holds. Since
$fmp(G)=1$, it follows that there exists an edge $e$ such that
$fmp(G-e)=0$. From Corollary \ref{cor2-4}, $(2)$ holds.
\end{proof}

\subsection{Graphs with large FMP number}

Wang et al. \cite{WangMaoChengZou} characterized even graphs with
given matching preclusion number.
\begin{theorem}{\upshape \cite{WangMaoChengZou}}\label{lem2-1}
Let $n,k$ be two integers with $n\geq 4k+6$, and let $G$ be an even
graph of order $n$. Then $mp(G)=n-k$ if and only if $\delta(G)=n-k$.
\end{theorem}

For an even graph $G$, we have $mp(G)\leq fmp(G)\leq \delta(G)$, and hence
the following corollary is immediate.
\begin{corollary}\label{cor2-5}
Let $n,k$ be two integers with $n\geq 4k+6$, and let $G$ be an even
graph of order $n$. Then $fmp(G)=n-k$ if and only if
$\delta(G)=n-k$.
\end{corollary}
\begin{proof}
Suppose $\delta(G)=n-k$, then by Theorem~\ref{lem2-1}, $mp(G)=n-k=\delta(G)$. Thus $fmp(G)=n-k$.
Conversely, we suppose $fmp(G)=n-k$. From $mp(G)\leq fmp(G)\leq \delta(G)$, we have $mp(G)\leq fmp(G)=n-k$ and $\delta(G)\geq n-k$. We want to show that $\delta(G)\leq n-k$. Otherwise, we can suppose that $\delta(G)>n-k$ and $\delta(G)=n-k+t=n-(k-t)$ where $t\geq 1$. Then $n\geq 4k+6>4(k-t)+6$, by
Theorem~\ref{lem2-1}, $mp(G)=\delta(G)=n-k+t$. Therefore $fmp(G)=n-k+t$, contradicting with $fmp(G)=n-k$. We obtain that $\delta(G)\leq n-k$, completing the proof.
\end{proof}

We now focus our attention on odd graphs.
\begin{proposition}\label{pro2-4}
Let $G$ be an odd graph of order $n\geq 7$. Then $fmp(G)=n-1$ if and
only if $G$ is a complete graph of order $n$.
\end{proposition}
\begin{proof}
From Theorem \ref{th1-2}, if $G$ is a complete graph of order $n$,
then $fmp(G)=n-1$. Conversely, we suppose $fmp(G)=n-1$. From
Observation \ref{obs1-1}, we have $\delta(G)\geq fmp(G)=n-1$, and
hence $G$ is a complete graph, as desired.
\end{proof}

\begin{proposition}\label{pro2-5}
Let $G$ be an odd graph of order $n\geq 8$. Then $fmp(G)=n-2$ if and
only if $\delta(G)=n-2$.
\end{proposition}
\begin{proof}
If $fmp(G)=n-2$, then $\delta(G)\geq fmp(G)=n-2$, and hence
$\delta(G)=n-2$ by Proposition \ref{pro2-4}. Conversely, if
$\delta(G)=n-2$, then $fmp(G)\leq n-2$. We need to show $fmp(G)\geq
n-2$. It suffices to prove that for every $F\subseteq E(G)$ and
$|F|=n-3$, $G-F$ has a fractional perfect matching. We first suppose
that $deg_{G[F]}(v)\leq \frac{n-5}{2}$ for every $v\in V(G)$, then
$$
deg_{G-F}(v)=deg_{G}(v)-deg_{G[F]}(v)\geq (n-2)-\frac{n-5}{2},
$$
and hence $\delta(G-F)\geq \frac{n}{2}$. From Theorem \ref{thA},
$G-F$ contains a Hamiltonian cycle, and hence there is a fractional
perfect matching in $G-F$. Next, we suppose that there exists a
vertex $v\in V(G)$ such that $deg_{G[F]}(v)\geq \frac{n-3}{2}$.
Since $deg_{G-F}(v)\geq deg_{G}(v)-|F|\geq 1$, it follows that there
exists a vertex $u\in V(G)$ such that $vu\in E(G-F)$. Let
$G_{1}=G-\{u,v\}$. Clearly, $|V(G_{1})|=n-2$ is odd, and $|F\cap
E(G_{1})|\leq n-3-\frac{n-3}{2}=\frac{n-3}{2}$, $\delta(G_{1})\geq
n-4$. If $deg_{G_{1}[F]}(x)\leq \frac{n-7}{2}$ for every $x\in
V(G_{1})$, then
$$
deg_{G_{1}-F}(v)\geq deg_{G_{1}}(v)-deg_{G_{1}[F]}(v)\geq
(n-4)-\frac{n-7}{2}=\frac{n-1}{2},
$$
and hence $\delta(G_{1}-F)> \frac{n-2}{2}$. By Theorem \ref{thA},
$G_{1}-F$ contains a Hamiltonian cycle, and hence there is a
fractional perfect matching in $G_{1}-F$, say $f'$. Clearly, $f'\cup
\{uv\}$ is a fractional perfect matching of $G-F$.

Suppose that there exists a vertex $s\in V(G_{1})$ such that
$deg_{G_{1}[F]}(s)\geq \frac{n-5}{2}$. Since $deg_{G_{1}-F}(s)\geq
deg_{G_{1}}(s)-|F|\geq 1$, it follows that there exists a vertex
$t\in V(G_{1})$ such that $st\in E(G_{1}-F)$. Let
$G_{2}=G_{1}-\{s,t\}$. Note that $|V(G_{2})|=n-4$ is odd, and
$|F\cap E(G_{2})|\leq n-3-\frac{n-3}{2}-\frac{n-5}{2}=1$,
$\delta(G_{2})\geq n-6$.

Observe that $G_{2}$ is a graph from $K_{n-4}$ by deleting at most
$\frac{n-5}{2}$ edges. Then $G_{2}-F$ is a graph from $K_{n-4}$
deleted at most $\frac{n-5}{2}+1=\frac{n-3}{2}<n-5$ edges for $n\geq
8$. By Theorem \ref{th1-3}, there is a fractional perfect matching
in $G_{2}-F$, say $f'$. Clearly, $f'\cup \{uv,st\}$ is a fractional
perfect matching of $G-F$.

We may now conclude that $fmp(G)=n-2$.
\end{proof}

\begin{theorem}\label{th2-7}
Let $n,k$ be two integers with $n\geq 4k+5$, and let $G$ be an odd
graph of order $n$. Then $fmp(G)=n-k$ if and only if
$\delta(G)=n-k$.
\end{theorem}
\begin{proof}
Suppose $\delta(G)=n-k$. Then $fmp(G)\leq \delta(G)=n-k$. We need to
show $fmp(G)\geq n-k$. It suffices to prove that for every
$F\subseteq E(G)$ and $|F|=n-k-1$, $G-F$ has a fractional perfect
matching. We first suppose that $deg_{G[F]}(v)\leq \frac{n-2k-1}{2}$
for any $v\in V(G)$, then
$$
deg_{G-F}(v)=deg_{G}(v)-deg_{G[F]}(v)\geq
(n-k)-\frac{n-2k-1}{2}=\frac{n+1}{2},
$$
and hence $\delta(G-F)\geq \frac{n+1}{2}$, from From Theorem
\ref{thA}, $G-F$ contains a Hamiltonian cycle, and hence there is a
fractional perfect matching in $G-F$. Next, we suppose that there
exists a vertex $v\in V(G)$ such that $deg_{G[F]}(v)\geq
\frac{n-2k+1}{2}$. Since $deg_{G-F}(v)\geq deg_{G}(v)-|F|\geq 1$, it
follows that there exists a vertex $u\in V(G)$ such that $vu\in
E(G-F)$. Let $G_{1}=G-\{u,v\}$. Clearly, $|V(G_{1})|=n-2$ is odd, and
$|F\cap E(G_{1})|\leq n-k-1-\frac{n-2k+1}{2}=\frac{n-5}{2}$. Since
$|F\cap E(G_{1})|\leq \frac{n-5}{2}$, it follows that for every vertex
pair $s,t\in V(G_{1})$, $deg_{G-F}(s)+deg_{G-F}(t)\geq
2(n-k-2)-\frac{n-5}{2}-1\geq n$ since $n\geq 4k+5$. So $G_{1}-F$
contains a Hamiltonian cycle, and hence there is a fractional
perfect matching in $G_{1}-F$, say $f'$. Clearly, $f'\cup\{uv\}$ is
a fractional perfect matching of $G-F$. From the above argument, we
conclude that $fmp(G)=n-k$.

Conversely, we suppose $fmp(G)=n-k$, we want to show that
$\delta(G)=n-k$. Furthermore, by induction on $k$, we prove that
$fmp(G)=n-k$ if and only if $\delta(G)=n-k$. From Proposition
\ref{pro2-4} and \ref{pro2-5}, the result follows for $k=1,2$.
Suppose that the argument is true for every integer $k'(k'<k)$, that
is, $fmp(H)=n-k'$ if and only if $\delta(H)=n-k'$. For integer $k$,
it follows from Observation \ref{obs1-1} that $\delta(G)\geq
fmp(G)=n-k$. We need to show that $\delta(G)=n-k$. Assume, on the
contrary, that $\delta(G)>n-k$. Let $\delta(G)=n-k+t=n-(k-t)$, where
$t\geq 1$. Since $k-t<k$, it follows from the induction hypothesis
that $fmp(G)=n-k+t<n-k$, which contradicts $fmp(G)=n-k$. So
$\delta(G)=n-k$.
\end{proof}

\section{Extremal problems}

We now consider the three extremal problems that we stated in the
Introduction.

\subsection{Results for $s(n,k)$ and $g(n,k)$}

For general $k$, we have the following result for $s(n,k)$.
\begin{theorem}\label{th3-1}
Let $n,k$ be two positive integers such that $n=\ell\cdot(k+1)+r$,
where $0\leq r\leq k$ and $6\leq k\leq n-1$.

$(1)$ If $(a)$ or $(b)$ or $(c)$ holds, then
$$
s(n,k)=\frac{nk}{2},
$$
where
\begin{itemize}
\item[] \ \ \ $(a)$ $r=0$;

\item[] \ \ \ $(b)$ $r\geq 1$, $k+r$ is odd;

\item[] \ \ \ $(c)$ $r\geq 1$, $k+r$ is even, $k$ is even, and $\ell\geq 2$.
\end{itemize}

$(2)$ If $r\geq 1$, $k+r$ is even, $k$ is odd, and $\ell\geq 2$,
then
$$
s(n,k)=\frac{nk+1}{2}.
$$

$(3)$ If $r\geq 1$, $k+r$ is even, $\ell=1$ and $k\leq n-2$, then
$$
\left \lceil\frac{nk}{2}\right \rceil \leq s(n,k)\leq
\frac{(n-1)(k+2)}{2}.
$$
\end{theorem}

\begin{proof}
If $r=0$, then let $G_n^1=\ell K_{k+1}$ and it follows from $(5)$ of
Observation \ref{obs1-1} and Theorem \ref{th1-3} that
$fmp(G_n^1)=k$, and hence $s(n,k)\leq \ell {k+1\choose
2}=\frac{nk}{2}$. Let $G$ be a graph of order $n$ and $fmp(G)=k$
such that $e(G)=|E(G)|$ is minimized. Then $\delta(G)\geq k$ and
$s(n,k)=e(G)\geq \frac{nk}{2}$. So $s(n,k)=\frac{nk}{2}$.

Since $k+r$ is odd, it follows that $k+r+1$ is even. Let $F_{1}$ be
a $k$-regular graph of order $k+r+1$ such that $F_{1}$ can be
decomposed into $k$ edge-disjoint perfect matchings, and let $F_{2}$
be the disjoint union of $(\ell-1)$ cliques of order $k+1$, that is,
$F_{2}=(\ell-1)K_{k+1}$. Let $G_n^1=F_{1}\cup F_{2}$. Then
$|V(G)|=(k+r+1)+(\ell-1)(k+1)=n$. Since $\delta(G_n^1)=k$, it
follows that $fmp(G_n^1)\leq \delta(G_n^1)=k$. Since $F_{1}$ can be
decomposed into $k$ edge-disjoint perfect matchings, it follows from
Observation \ref{obs1-1} that $fmp(F_1)\geq mp(F_1)\geq k$. From
$(5)$ of Observation \ref{obs1-1},
$fmp(G_n^1)=\min\{fmp(F_1),fmp(K_{k+1})\}=k$. Clearly,
$e(G)=\lceil\frac{nk}{2}\rceil$, implying $s(n,k)\leq
\lceil\frac{nk}{2}\rceil$. Since $fmp(G)=k$, it follows that
$\delta(G)\geq fmp(G)\geq k$, and hence $s(n,k)\geq \frac{nk}{2}$,
completing the proof for the case $k+r$ being odd.

Suppose $k+r$ is even. Then $k+r+1$ is odd. We define four
graphs:
\begin{itemize}
\item Let $F_3$ be a $k$-regular graph of order $k+r$ such that $F_{3}$
can be decomposed into $k$ edge-disjoint perfect matchings;

\item Let $F_{4}=(\ell-2)K_{k+1}$;

\item Let $F_{5}$ be a $k$-regular graph of order $k+2$ such that $F_{5}$
can be decomposed into $k$ edge-disjoint perfect matchings, where
$k$ is even;

\item Let $F_{6}$ be a graph obtained from $K_{k+2}$ by deleting a maximum matching, where $k$ is odd.
\end{itemize}

Suppose that $k$ is even and $\ell\geq 2$. Let $G_n^2=F_{3}\cup
F_{4}\cup F_{5}$. Then $|V(G)|=(k+r)+(\ell-2)(k+1)+(k+2)=n$. Since
$\delta(G_n^2)=k$, it follows that $fmp(G_n^2)\leq \delta(G_n^2)=k$.
Since $F_{5}$ can be decomposed into $k$ edge-disjoint perfect
matchings, it follows from Observation \ref{obs1-1} that
$fmp(F_5)\geq mp(F_5)\geq k$. From $(5)$ of Observation
\ref{obs1-1},
$fmp(G_n^2)=\min\{fmp(F_3),fmp(K_{k+1}),fmp(F_{5})\}=k$. Clearly,
$e(G)=\frac{nk}{2}$, implying $s(n,k)\leq \frac{nk}{2}$. Since
$fmp(G)=k$, it follows that $\delta(G)\geq fmp(G)\geq k$, and hence
$s(n,k)\geq \frac{nk}{2}$, completing the proof for $k$ is even and $\ell\geq 2$.

Suppose that $k$ is odd and $\ell\geq 2$. Let $G_n^3=F_{3}\cup
F_{4}\cup F_{6}$. Then $|V(G)|=(k+r)+(\ell-2)(k+1)+(k+2)=n$. Since
$\delta(G_n^3)=k$, it follows that $fmp(G_n^3)\leq \delta(G_n^3)=k$.
From Proposition \ref{pro2-5}, we have $fmp(F_6)=k$. From $(5)$ of
Observation \ref{obs1-1},
$fmp(G_n^3)=\min\{fmp(F_3),fmp(K_{k+1}),fmp(F_{6})\}=k$. Clearly,
$e(G)=\frac{nk+1}{2}$, implying $s(n,k)\leq \frac{nk+1}{2}$. Since
$fmp(G)=k$, it follows that $\delta(G)\geq fmp(G)\geq k$. Since
$n,k,r$ are all odd integers, it follows that $s(n,k)\geq
\frac{nk+1}{2}$, completing the proof for $k$ is odd and $\ell\geq 2$.

Suppose that $r\geq 1$, $k+r$ is even, $\ell=1$ and $k\leq n-2$.
Since $fmp(G)=k$, it follows that $\delta(G)\geq fmp(G)\geq k$, and
hence $s(n,k)\geq \lceil\frac{nk}{2}\rceil$. Let $F_7$ be a
$k$-regular graph of order $n-1$ such that $F_{7}$ can be decomposed
into $k$ edge-disjoint perfect matchings, and let $G_n^4$ be the
graph obtained from $F_7$ by adding a new vertex $v$ and then adding
all edges from $v$ to $F_7$.
\begin{claim}\label{Clm:1a}
$fmp(G_n^4)\geq k$.
\end{claim}
\noindent \emph{Proof.}
Assume, to the contrary, that $fmp(G_n^4)=k-1$. For any $X\subseteq
E(G_n^4)$ and $|X|=k-1$, we suppose $|X\cap E(F_{7})|=a$. Since
there are $k$ edge-disjoint perfect matchings in $F_7$, it follows
that there are $(k-a)$ edge-disjoint perfect matchings in $F_7-X$,
and hence $\delta(F_7-X)\geq k-a$. Let $F_7'$ be the graph obtained
from $(k-a)$ edge-disjoint perfect matchings. Since $k\leq n-2$, it
follows that there is at least one edge, say $vu$, from $v$ to
$F_{7}$ in $G_n^4-X$. Then $|N_{F_7'}(u)|=k-a$ and there is at least
one edge, say $vw$, from $v$ to $N_{F_7'}(u)$. Then $uvwu$ is a
triangle, and $uw$ is an edge of a perfect matching $M$ of $F_7'$.
It is clear that $(M-uw)\cup uvwu$ is fractional perfect matching in
$F_7'$, and hence $fmp(G_n^4)\geq k$, a contradiction. \hfill $\Diamond$

By Claim~\ref{Clm:1a}, we have $fmp(G_n^4)\geq k$. Then $s(n,k)\leq
e(G)=\frac{(n-1)(k+2)}{2}$, as desired.
\end{proof}

Note that $g(n,k)=s(n,k+1)-1$. Then we have the following result for
$g(n,k)$.

\begin{corollary}\label{cor3-2}
Let $n,k$ be two positive integers such that $n=\ell\cdot(k+2)+r$,
where $0\leq r\leq k+1$ and $k\geq 5$.

$(1)$ If $(a)$ or $(b)$ or $(c)$ holds, then
$$
g(n,k)=\frac{n(k+1)}{2}-1,
$$
where $(a)$ $r=0$; $(b)$ $r\geq 1$, $k+r+1$ is odd; $(c)$ $r\geq 1$,
$k+r+1$ is even, $k$ is odd, and $\ell\geq 2$.

$(2)$ If $r\geq 1$, $k+r+1$ is even, $k$ is even, and $\ell\geq 2$,
then
$$
g(n,k)=\frac{n(k+1)+1}{2}-1.
$$

$(3)$ If $r\geq 1$, $k+r+1$ is even, $\ell=1$ and $k\leq n-3$, then
$$
\left \lceil\frac{n(k+1)}{2}\right \rceil-1 \leq g(n,k)\leq
\frac{(n-1)(k+3)}{2}-1.
$$
\end{corollary}

For $k=0,1$, we have the following result.
\begin{proposition}\label{pro3-1}
Let $n,k$ be two positive integers such that $n\geq 3$ is odd. Then

$(1)$ $s(n,0)=0$;

$(2)$ $s(n,1)=\frac{n+3}{2}$.
\end{proposition}
\begin{proof}
$(1)$ Let $H_1$ be the graph of order $n$ with no edges. Clearly,
$fmp(H_1)=0$. Then $s(n,0)\leq 0$, and so $s(n,0)=0$.

$(2)$ Let $H_2$ be a graph of order $n$ with $\frac{n-3}{2}$
independent edges and a triangle, that is, $H_2=\frac{n-3}{2}K_2\cup
K_3$. From Theorem \ref{th2-3}, we have $fmp(H_2)\geq 1$. By
deleting one edge in $\frac{n-3}{2}K_2$, there is no FMP in the
resulting graph, and so $fmp(H_2)=1$. Then $s(n,1)\leq
\frac{n+3}{2}$. Conversely, let $G$ be an odd graph of order $n$
with $fmp(G)=1$. Since $G$ contains a FMP, it follows from Theorem
\ref{th2-3} that $G$ has at least $\frac{n+3}{2}$ edges, and hence
$s(n,1)\geq \frac{n+3}{2}$. So $s(n,1)=\frac{n+3}{2}$.
\end{proof}

\begin{theorem}\label{lem3-3}
Let $n,k$ be two positive integers such that $n\geq 9$ is odd. Then
$$
s(n,2)=n+3.
$$
\end{theorem}
\begin{proof}
Let $H_3$ be an odd graph of order $n$. If $n=4k+1$, let $H_3=F_1\cup
(k-2)C_4$, where $F_1$ is a graph obtained from two cycles of order
$4$, say $C_4^1=u_1u_2u_3u_4u_1,C_4^2=v_1v_2v_3v_4v_1$, by adding a
new $w$ and new edges in $\{wu_1,wu_4,wv_1,wv_4\}$; see Figure 1. By
deleting two edges in $\{wu_4,wv_4\}$, we have no FMP in the
resulting graph by Theorem \ref{th2-3}, and hence $fmp(H_3)\leq 2$.
For any edge $e$ in $H_3$, $H_3-e$ contains a $(2k-4)K_2\cup C_5\cup
C_4$ as its subgraph, and hence $fmp(H_3)\geq 2$. So $fmp(H_3)=2$
and $s(n,2)\leq n+3$.
\begin{figure}[!hbpt]
\begin{center}
\includegraphics[scale=0.7]{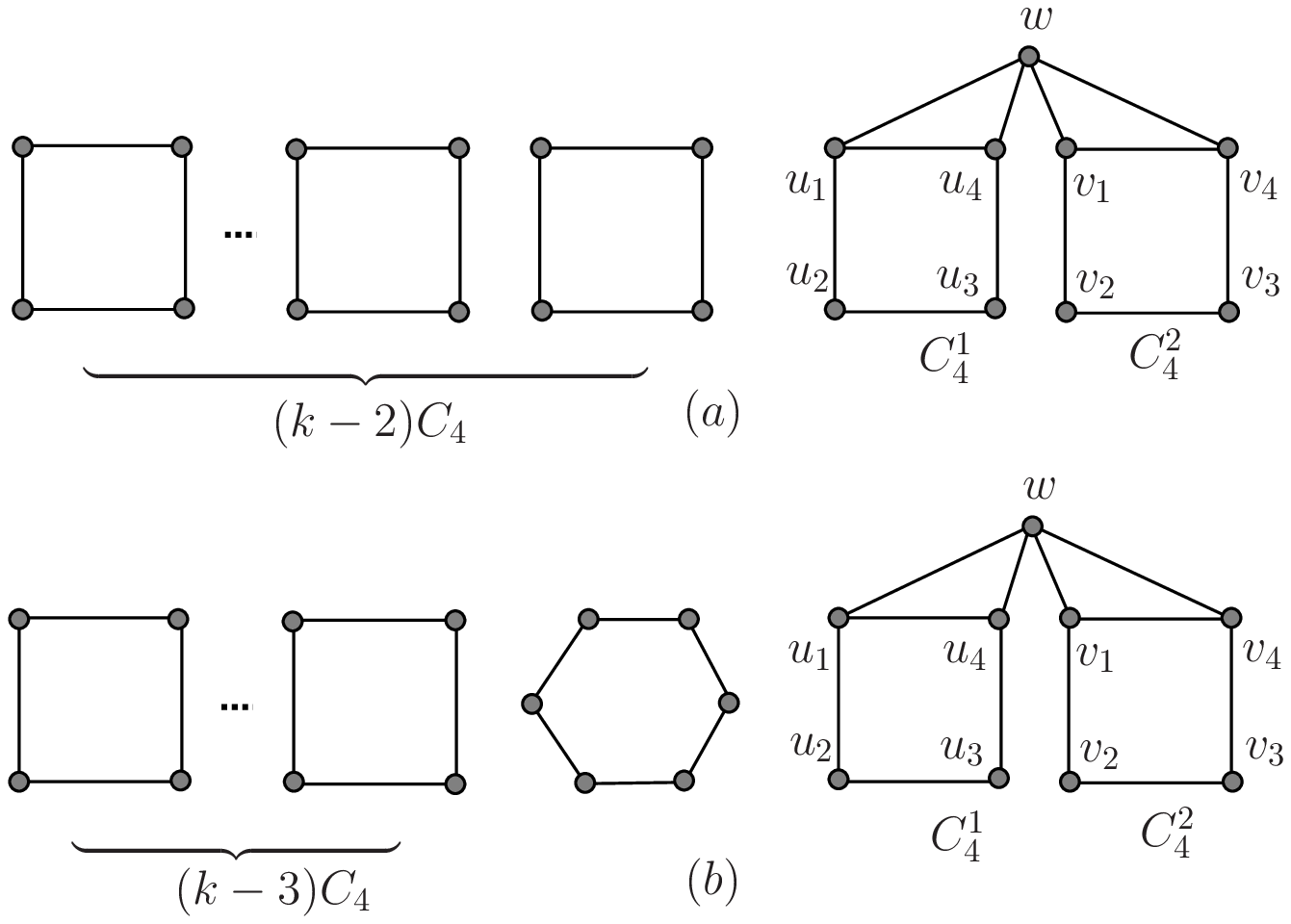}
\end{center}
\begin{center}
Figure 1: Graphs $H_3$ and $H_4$.
\end{center}\label{fig1}
\end{figure}

If $n=4k+3$, let $H_4=F_1\cup (k-3)C_4\cup C_6$, where $F_1$ is a graph
obtained from two cycles of order $4$, say
$C_4^1=u_1u_2u_3u_4u_1,C_4^2=v_1v_2v_3v_4v_1$, by adding a new $w$
and new edges in $\{wu_1,wu_4,wv_1,wv_4\}$; see Figure 1. By
deleting two edges in $\{wu_4,wv_4\}$, we have no FMP in the
resulting graph by Theorem \ref{th2-3}, and hence $fmp(H_4)\leq 2$.
For any edge $e$ in $H_4$, $H_4-e$ contains a $(2k-4)K_2\cup C_5\cup
C_4$ as its subgraph, and hence $fmp(H_4)\geq 2$. So $fmp(H_4)=2$
and $s(n,2)\leq n+3$.

It suffices to show $s(n,2)> n+2$. Suppose $s(n,2)\leq n+2$. Then
there exists an odd graph $G$ of order $n$ with $e(G)=n+2$ and
$fmp(G)=2$. It is clear that $\delta(G)\geq fmp(G)=2$.

\begin{claim}\label{Clm:1}
$\Delta(G)\leq 6$.
\end{claim}
\noindent \emph{Proof.}
Assume, on the contrary, that $\Delta(G)\geq 7$. Then there exists a
vertex of $G$, say $v$, such that $deg_{G}(v)\geq 7$.
\begin{eqnarray*}
e(G)&=&\frac{1}{2}\sum_{u\in
V(G)}deg_{G}(u)=\frac{1}{2}\left(deg_{G}(v)+\sum_{u\in V(G), \ u\neq
v}deg_{G}(u)\right)\\
&\geq&\frac{1}{2}(7+2(n-1))=\frac{1}{2}(2n+5)>n+2,
\end{eqnarray*}
a contradiction. \hfill $\Diamond$

From Claim \ref{Clm:1}, we have $\Delta(G)\leq 6$.
\begin{fact}\label{Clm:3}
If there exists a connected component of $G$, say $C$,
$\Delta(C)=2$, then $|V(C)|$ is even.
\end{fact}
We assume that there exists a connected component $C$ of $G$ such
that $\Delta(C)\leq 6$. We may assume that $C=G$, and this
assumption does not affect the validity of our proof. We distinguish
the following cases to show this theorem.

\textbf{Case 1.} $\Delta(G)=6$.

In this case, there exists a vertex of degree $6$ in $G$, say $u$.
Since $e(G)=n+2$, it follows that any vertex in $G-u$ has degree
$2$, and hence $G=A_1$ must be a graph obtained from $3$ cycles by
sharing exactly one vertex; see Figure 2 $(a)$. Clearly, $G-u$ is
the union of three paths, say $P^1=u_1u_2\ldots u_{\ell_1}$,
$P^2=v_1v_2\ldots v_{\ell_2}$ and $P^3=w_1w_2\ldots w_{\ell_3}$.
\begin{figure}[!hbpt]
\begin{center}
\includegraphics[scale=0.7]{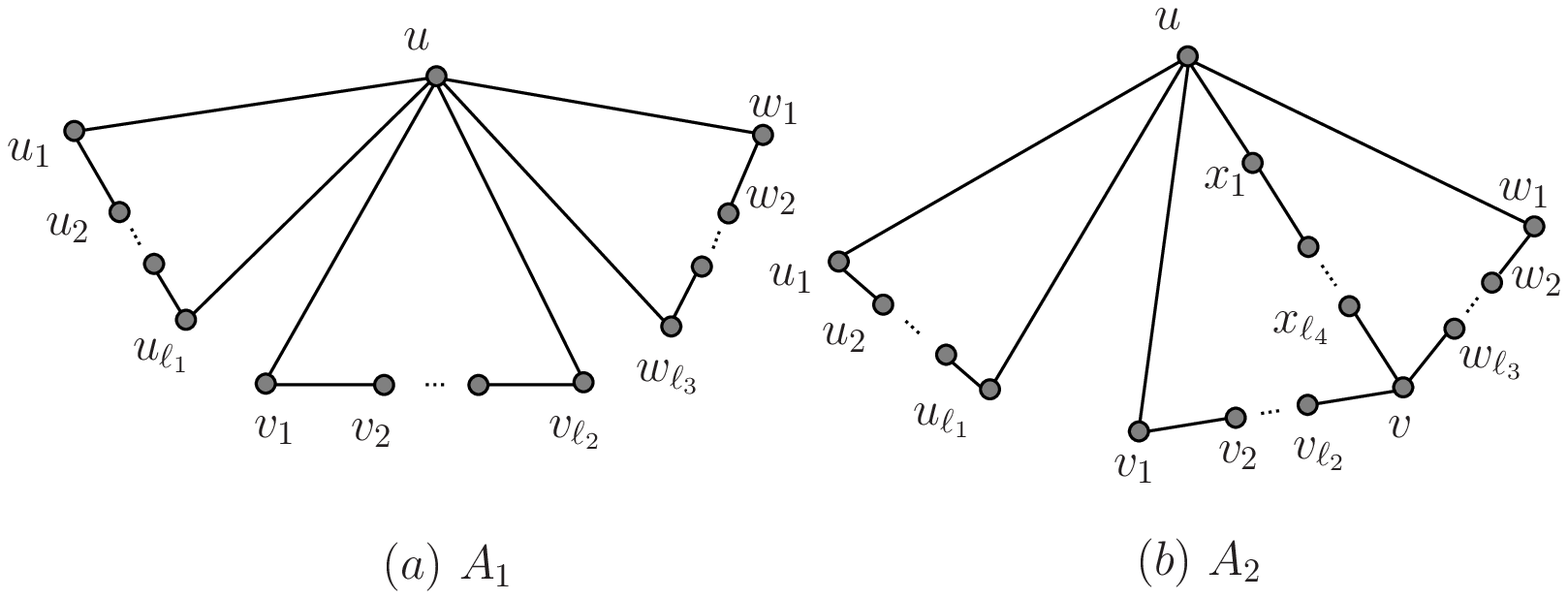}
\end{center}
\begin{center}
Figure 2: Graphs $A_1$ and $A_2$.
\end{center}\label{fig2}
\end{figure}

Since $|V(G)|$ is odd, it follows that $\ell_1+\ell_2+\ell_3$ is
even, and hence $\ell_1$ or $\ell_2$ or $\ell_3$ is even. Without
loss of generality, suppose $\ell_1$ is even. From Theorem
\ref{th2-3}, $G-u_1u_2$ has no FMP, implying $fmp(G)\leq 1$,
contradicting to the fact that $fmp(G)=2$.

\textbf{Case 2.} $\Delta(G)=5$.

Since $e(G)=n+2$, it follows that there exist two vertices $u,v$
such that $deg_G(u)=5$ and $deg_G(v)=3$. Then $G=A_2$, as shown in
Figure 2 $(b)$, where $\ell_{1},\ell_{2},\ell_{3},\ell_{4}$ are all
integers and $\ell_{1}+\ell_{2}+\ell_{3}+\ell_{4}=n-2$. Since
$|V(G)|$ is odd, it follows that
$\ell_{1}+\ell_{2}+\ell_{3}+\ell_{4}$ is odd. We consider the
following two cases by the value of $\ell_{1}$.

If $\ell_{1}$ is odd, then $\ell_{2}+\ell_{3}+\ell_{4}$ is even, and
hence $\ell_{2}$ or $\ell_{3}$ or $\ell_{4}$ is even. If $\ell_{2}$
is even, then it follows from Theorem \ref{th2-3} that
$G-v_{\ell_{2}}v$ has no FMP, and hence $fmp(G)\leq 1$,
contradicting to $fmp(G)=2$. If $\ell_{3}$ is even, then it follows
from Theorem \ref{th2-3} that $G-w_{\ell_{3}}v$ has no FMP, and hence
$fmp(G)\leq 1$, contradicting to $fmp(G)=2$. If $\ell_{4}$ is even,
then it follows from Theorem \ref{th2-3} that $G-x_{\ell_{4}}v$ has
no FMP, and hence $fmp(G)\leq 1$, contradicting to $fmp(G)=2$.

If $\ell_{1}$ is even, then it follows from Theorem \ref{th2-3} that
$G-(u_{1}u_{2})$ has no FMP, and hence $fmp(G)\leq 1$, contradicting
to $fmp(G)=2$.

\textbf{Case 3.} $\Delta(G)=4$.

Since $e(G)=n+2$, it follows that $G$ has at most two vertices of
degree $4$. Then $G=B_1$ or $G=B_2$ or $G=B_3$ or $G=B_4$; see Figure 3.

\begin{figure}[!hbpt]
\begin{center}
\includegraphics[scale=0.8]{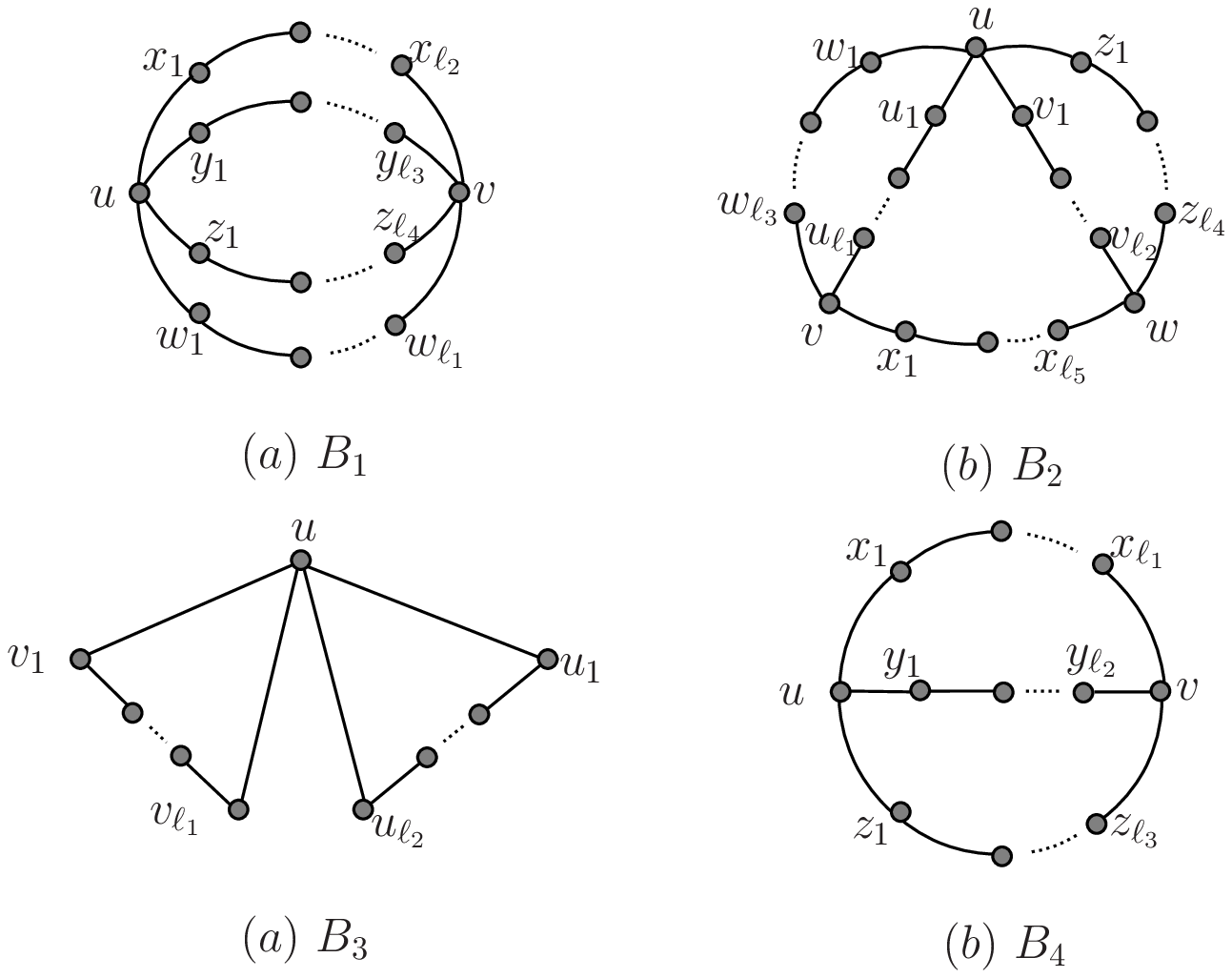}
\end{center}
\begin{center}
Figure 3: Graphs $B_1,B_2,B_3,B_4$.
\end{center}\label{fig4}
\end{figure}

Suppose $G=B_1$. Then there are two vertices of degree $4$, say
$u,v$. If $uv\in E(G)$, then $\ell_4=0$ and
$\ell_{1}+\ell_{2}+\ell_{3}$ is odd. Then
$\ell_{1},\ell_{2},\ell_{3}$ are all odd, or we can assume that
$\ell_{1},\ell_{2}$ are even and $\ell_{3}$ is odd. If
$\ell_{1},\ell_{2},\ell_{3}$ are all odd, then there exists some
$\ell_j$ such that $\ell_j\geq 3$. Without loss of generality, let
$\ell_1\geq 3$. From Theorem \ref{th2-3}, $G-w_1w_{2}$ has no FMP,
and hence $fmp(G)\leq 1$, contradicting to $fmp(G)=2$. If
$\ell_{1},\ell_{2}$ are even and $\ell_{3}$ is odd, then it follows
from Theorem \ref{th2-3} that $G-w_1w_{2}$ has no FMP, and hence
$fmp(G)\leq 1$, contradicting to $fmp(G)=2$. Suppose $uv\notin
E(G)$. Then $\ell_4\geq 1$ and each vertex of $G-u-v$ has degree $2$
and hence $G=B_1$; see Figure 3 $(a)$. Since $|V(G)|$ is odd, it
follows that $\ell_{1}+\ell_{2}+\ell_{3}+\ell_{4}$ is odd, and hence
$\ell_{1}$ or $\ell_{2}$ or $\ell_{3}$ or $\ell_{4}$ is odd. Without
loss of generality, let $\ell_{4}$ is odd. Then
$\ell_{1}+\ell_{2}+\ell_{3}$ is even, and hence $\ell_{1}$ or
$\ell_{2}$ or $\ell_{3}$ is even. Without loss of generality, let
$\ell_{1}$ is even. From Theorem \ref{th2-3}, $G-w_1w_{2}$ has no
FMP, and hence $fmp(G)\leq 1$, also a contradiction.

Suppose $G=B_2$. Then there is only one vertex of degree $4$, say
$u$. Then there are two vertices of degree $3$, say $v,w$, and each
vertex in $G-u-v-w$ has degree $2$. Therefore, $G=B_2$; see Figure 3
$(b)$. Suppose $\ell_i\geq 1$ for each $i \ (1\leq i\leq 5)$. If
$\ell_1$ and $\ell_3$ are odd, then $G-z_1z_2$ or $G-z_{\ell_4}w$
has no FMP, and hence $fmp(G)\leq 1$, a contradiction. If $\ell_1$
is even and $\ell_3$ is odd, then $G-u_1u_2$ has no FMP, and hence
$fmp(G)\leq 1$, a contradiction. If $\ell_1$ and $\ell_3$ are even,
then $\ell_2$ and $\ell_4$ are even, and hence $\ell_5$ is even. So
$G-z_1z_2$ has no FMP, and hence $fmp(G)\leq 1$, a contradiction. If
there exists some $\ell_j$ such that $\ell_j=0$, then we can get a
contradiction by similar method as the case $uv\in E(G)$ in $B_1$.

Suppose $G=B_3$. Then $\ell_1+\ell_2$ is even. If $\ell_1$ and
$\ell_2$ is even, $G-v_1v_2$ has no FMP, and hence $fmp(G)\leq 1$, a
contradiction. If $\ell_1$ and $\ell_2$ is odd, $G-uv_1$ has no FMP,
and hence $fmp(G)\leq 1$, a contradiction.

Suppose $G=B_4$. Similarly to the proof of $B_1$, we can also get a
contradiction.

\textbf{Case 4.} $\Delta(G)=3$.

In this case, $G$ has four vertices of degree $3$ and each of the
other $n-4$ vertices has degree $2$. Then $G=D_1$ or $G=D_2$ or
$G=D_3$ or $G=D_4$; see Figure 4.

\begin{figure}[!hbpt]
\begin{center}
\includegraphics[scale=0.8]{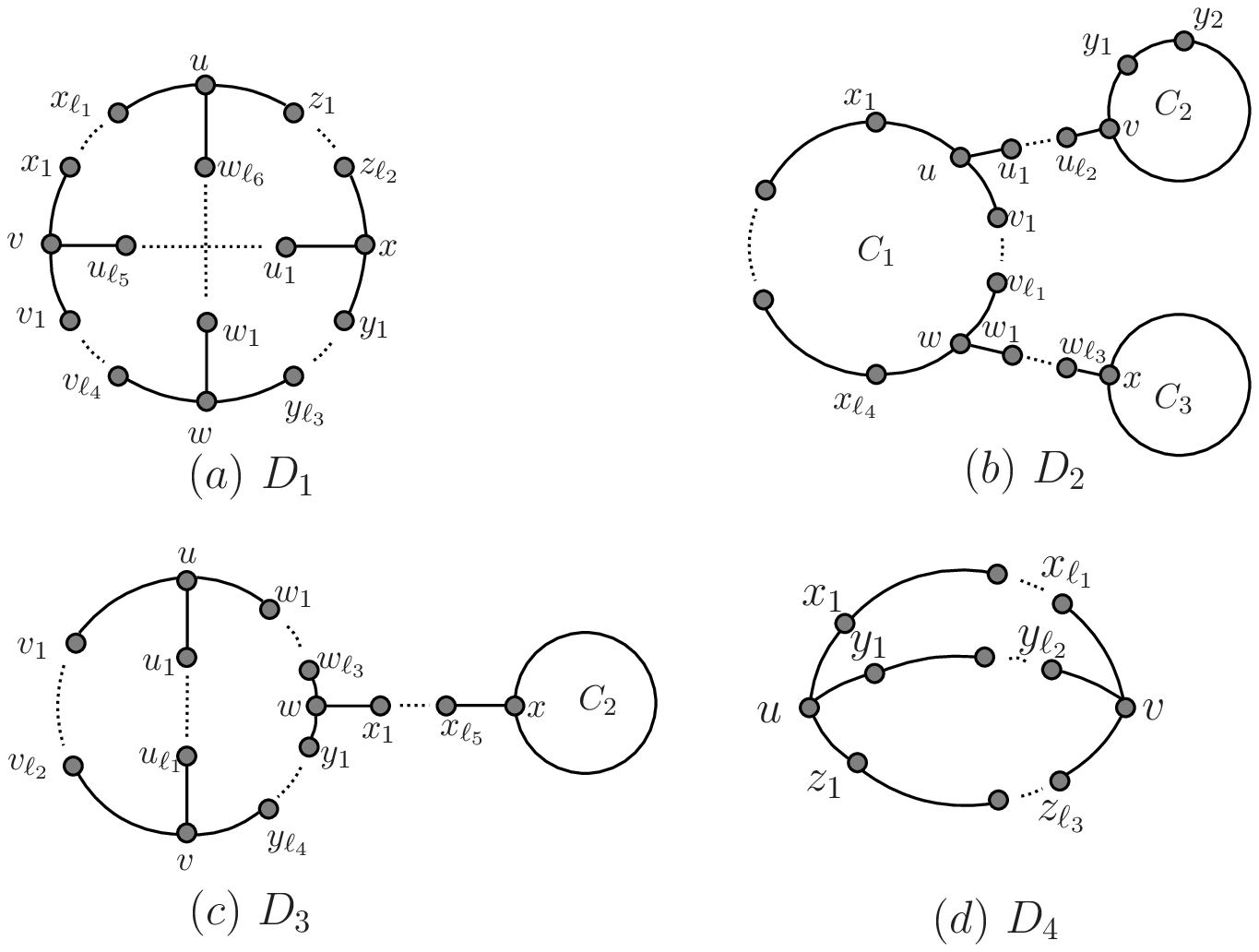}
\end{center}
\begin{center}
Figure 4: Graphs $D_1,D_2,D_3,D_4$.
\end{center}\label{fig5}
\end{figure}

Suppose $G=D_1$. If $\ell_{5}$ is odd and $\ell_6$ is even, then
$\ell_{1}+\ell_{2}+\ell_{3}+\ell_{4}$ is even.
\begin{fact}\label{Clm:3}
$(1)$ At least one in $\{\ell_i\,|\,1\leq i\leq 4\}$ is even.

$(2)$ At least one in $\{\ell_i\,|\,1\leq i\leq 4\}$ is odd.
\end{fact}

Without loss of generality, we suppose $\ell_1$ is odd. Since
$\ell_{1}+\ell_{2}+\ell_{3}+\ell_{4}$ is even and at least one of
one in $\{\ell_i\,|\,1\leq i\leq 4\}$, it follows that there exists
some $\ell_j \ (j=2,3,4)$ such that $\ell_j$ is odd. If $\ell_1$ and
$\ell_2$ is odd, then $\ell_3$ and $\ell_4$ is even, and hence
$G-xu_1$ has no FMP, implying that $fmp(G)\leq 1$, a contradiction.
If $\ell_1$ is odd, and either $\ell_3$ or $\ell_4$ is even, then
$G-xu_1$ has no FMP, implying that $fmp(G)\leq 1$, a contradiction.

Suppose that $\ell_{5}$ is odd, $\ell_6$ is odd, and
$\ell_{1}+\ell_{2}+\ell_{3}+\ell_{4}$ is odd. Then the number of odd
integer in $\{\ell_i\,|\,1\leq i\leq 4\}$ is $1$ or $3$. If the
number is $1$, say $\ell_1$, then $G-xu_1$ has no FMP, implying that
$fmp(G)\leq 1$, a contradiction. If the number is $3$, say
$\ell_1,\ell_2,\ell_3$, then $G-xu_1$ has no FMP, implying that
$fmp(G)\leq 1$, a contradiction.

Suppose that $\ell_{5}$ is even and $\ell_6$ is even. Then
$\ell_{1}+\ell_{2}+\ell_{3}+\ell_{4}$ is odd. By the above proof,
the number of odd integer in $\{\ell_i\,|\,1\leq i\leq 4\}$ is $1$
or $3$. Then $G-ux_{\ell}$ has no FMP, implying that $fmp(G)\leq 1$,
a contradiction.

Suppose $G=D_2$. Let $c_{i}$ be the length of $C_{i} \ (i=1,2,3)$,
that is, $|V(C_{i})|=c_i$.
\begin{claim}\label{Clm:5}
For $i=2,3$, $|V(C_{i})|$ are even.
\end{claim}
\noindent \emph{Proof.}
Assume, on the contrary, that $|V(C_{2})|$ is odd. Then
$G-y_{1}y_{2}$ has no FMP, and hence $fmp(G)\leq 1$, contradicting
to the fact $fmp(G)=2$. \hfill $\Diamond$

\begin{claim}\label{Clm:6}
$\ell_{2}=\ell_{3}=0$.
\end{claim}
\noindent \emph{Proof.}
Assume, on the contrary, that $\ell_{2}>0$. Then
$G-u_{\ell_{2}-1}u_{\ell_{2}}$ has no FMP, and hence $fmp(G)\leq 1$,
contradicting to the fact $fmp(G)=2$. \hfill $\Diamond$

Since $|V(G)|$ is odd, it follows that $|V(C_{1})|$ is odd. Then
$G-ux_{1}$ has no FMP, and hence $fmp(G)\leq 1$, contradicting to
the fact $fmp(G)=2$.

Suppose $G=D_3$. From the above proof, we have $\ell_{5}=0$ and
$|V(C_{2})|$ is even. Then $\ell_{1}+\ell_{2}+\ell_{3}+\ell_{4}$ is
even. If both $\ell_{1}$ and $\ell_{2}$ is even, then
$\ell_{3}+\ell_{4}$ is even, and hence $G-uw_{1}$ has no FMP, and so
$fmp(G)\leq 1$, contradicting to $fmp(G)=2$. If both $\ell_{1}$ and
$\ell_{2}$ are odd, then $\ell_{3}+\ell_{4}$ is even, and hence
$G-uv_{1}$ has no FMP, and so $fmp(G)\leq 1$, a contradiction. If
$\ell_{1}$ is even and $\ell_{2}$ is odd, then $\ell_{3}+\ell_{4}$
is odd, and hence $G-uv_{1}$ has no FMP, and hence $fmp(G)\leq 1$,
also a contradiction.

Suppose $G=D_4$. Similarly to the proof of $B_1$, we can also get a
contradiction.
\end{proof}

\subsection{Results for $f(n,k)$}

Next, we give the exact value of $f(n,k)$.
\begin{theorem}\label{thm4-2}
Let $n,k$ be two positive integers. If $1\leq k\leq n-1$, then
$$
f(n,k)=\binom{n-1}{2}+k.
$$
\end{theorem}
\begin{proof}
 Let $G$ be a graph with $n$ vertices such that
$|E(G)|\geq \binom{n-1}{2}+k$. Clearly, $|E(\overline{G})|\leq
n-k-1$. For any $F\subseteq E(G)$ and $|F|=k-1$, we have
$|E(\overline{G-F})|\leq n-2$. Since $fmp(K_{n})=n-1$, it follows
that $G-F$ has a fractional perfect matching, and hence $fmp(G)\geq
k$. So $f(n,k)\leq \binom{n-1}{2}+k$.
 To show $f(n,k)\geq
\binom{n-1}{2}+k$, we construct $G_k$ obtained from $K_{n-1}$ by
adding a new vertex $v$ and then adding $(k-1)$ edges from $v$ to
$K_{n-1}$. It is clear that $G_k$ is a connected graph on $n$
vertices, $|E(G_k)|=\binom{n-1}{2}+k-1$, and $mp(G_k)<k$. So
$f(n,k)=\binom{n-1}{2}+k$.
\end{proof}


\end{document}